\documentclass[10pt,article,twocolumn]{IEEEtran}

\IEEEoverridecommandlockouts
\usepackage{graphicx}
\usepackage{amsmath}
\usepackage{hyperref}

\newtheorem{theorem}{Theorem}[section]
\newtheorem{corollary}{Corollary}[section]
\newtheorem{lemma}{Lemma}[section]
\newtheorem{algorithm}{Algorithm}[section]

\usepackage{mathptmx}       
\usepackage{helvet}         
\usepackage{courier}        
\usepackage{type1cm}        
\usepackage{makeidx}         
\usepackage{graphicx}        
\usepackage{multicol}        
\usepackage[bottom]{footmisc}

\RequirePackage[OT1]{fontenc}
\RequirePackage{amsmath,amssymb}

\numberwithin{equation}{section}

\newcommand{\mB}[1]{{\mathbb{#1}}}

\newcommand{\B}[1]{{\bf #1}}
\newcommand{\R}[1]{{\rm #1}}
\newcommand{\T}[1]{{\tt #1}}

\begin{document}

\title{Kalman smoothing  and block tridiagonal systems: new connections and
numerical stability results}
\author{Aleksandr Y. Aravkin, Bradley M. Bell, James V. Burke, Gianluigi Pillonetto
\thanks{Aleksandr Y. Aravkin
(saravkin@us.ibm.com)
is with the IBM T.J. Watson Research Center,
Yorktown Heights, NY 10598}
\thanks{Bradley M. Bell
(bradbell@apl.washington.edu)
is with Applied Physics Lab University of Washington, Seattle, WA}
\thanks{James V. Burke
(burke@math.washington.edu)
is with Department of Mathematics University of Washington, Seattle, WA}
\thanks{Gianluigi Pillonetto
(giapi@dei.unipd.it)
is with Dipartimento di Ingegneria dell'Informazione,
University of Padova, Padova, Italy.}}

\maketitle
\thispagestyle{empty}
\pagestyle{empty}

\begin{abstract}

The Rauch-Tung-Striebel (RTS) and the Mayne-Fraser (MF) algorithms are two
of the most popular smoothing schemes to reconstruct the state of
a dynamic linear system from measurements collected on a fixed interval.
Another (less popular) approach is the Mayne (M) algorithm introduced
in his original paper under the name of Algorithm A.
In this paper, we analyze these three smoothers from an optimization and
algebraic perspective,
revealing new insights on their numerical
stability properties. In doing this, we re-interpret classic recursions as matrix
decomposition methods
for block tridiagonal matrices.\\
First, we show that the classic RTS smoother is an implementation of the
forward block tridiagonal (FBT) algorithm (also known as Thomas algorithm) for particular block tridiagonal
systems.
We study the numerical stability properties of this scheme, 
connecting the condition number of the full system to properties of the
individual blocks encountered during standard recursion.\\
Second, we study the M smoother, and prove it is equivalent to a
backward block tridiagonal
(BBT) algorithm with 
a stronger stability guarantee than RTS.\\
Third, we illustrate how the MF smoother solves a block tridiagonal system, and prove
 that it has the same numerical stability properties of RTS (but not those of M).\\
Finally, we present a new hybrid RTS/M (FBT/BBT) smoothing scheme,
which is faster than MF, and has the same numerical stability guarantees of RTS and MF.
\end{abstract}

\section{Introduction}

Kalman filtering and smoothing methods form a broad category 
of computational algorithms used for inference on 
noisy dynamical systems. Since their invention~\cite{kalman} 
and early development~\cite{KalBuc},
these algorithms have become a gold standard in 
a range of applications, including space exploration, 
missile guidance systems, general 
tracking and navigation, and weather prediction. 
Numerous books and papers have been written on these methods and their extensions, 
addressing modifications for use in nonlinear systems~\cite{Grew2007}, 
robustification against unknown models~\cite{Pet},
smoothing data over time intervals~\cite{Anderson:1979,YAA}. Other studies regard 
Kalman smoothing with unknown parameters \cite{Bell2000}, 
constrained Kalman smoothing \cite{Bell2009},
and robustness against bad measurements~\cite{Durovic1999,Meinhold1989,Cipra1997,Schick1994,Masreliez1977,Kirlin1986,Kassam1985,AravkinIeee2011}
and sudden changes in state
\cite{SYSID2012tks,AravkinBurkePillonetto2013,Farahmand2011}. 

In this paper, we consider a linear Gaussian model  
specified as follows: 
%
\begin{equation}
\label{IntroGaussModel}
\begin{array}{rcll}
	{x_1}&= & x_0+{w_1},
	\\
	{x_k} & = & G_k {x_{k-1}}  + {w_k}& k = 2 , \ldots , N,
	\\
	{z_k} & = & H_k {x_k}      + {v_k}& k = 1 , \ldots , N\;,
\end{array}
\end{equation}
where $x_0$ is known,  ${x_k}, {w_k} \in \mB{R}^n$, ${z_k}, {v_k} \in \mB{R}^{m(k)}$, 
(measurement dimensions can vary between time points) $G_k \in \mB{R}^{n\times n}$ and $H_k \in \mB{R}^{m(k)\times n}$. 
Finally, ${w_k}$, ${v_k}$ are mutually independent zero-mean Gaussian random variables
with known positive definite covariance matrices $Q_k$ and $R_k$, respectively.
Extensions have been proposed that work with singular $Q_k$ and $R_k$ 
(see e.g.~\cite{Bierman1983,KitagawaGersh,Chui2009}),
but in this paper we confine our attention to the nonsingular case. \\


\paragraph{Kalman smoothing algorithms} 

To obtain the minimum variance estimates of the states given the full 
measurement sequence $\{z_1, \dots, z_N\}$, two of the most popular Kalman smoothing schemes
are the Rauch-Tung-Striebel (RTS) and the Mayne-Fraser (MF) algorithm based on the two-filter formula.\\
RTS was derived in \cite{RTS} and computes the state estimates
by means of forward-backward recursions which are sequential in nature 
(first run forward, then run backward). 
An elegant and simple derivation 
of RTS using projections onto spaces spanned by suitable random variables
can be found in \cite{Ansley1982}.\\
MF was proposed by \cite{Mayne1966,FraserPotter1969}.
It computes the smoothed estimate as a 
combination of forward and backward Kalman filtering estimates, which can be run independently
(and in particular, in parallel).  
A derivation of MF
from basic principles where the backward recursions are related to  
maximum likelihood state estimates can be found in \cite{Willsky1981}.\\
A third algorithm we study (and dub M)
was also proposed by Mayne. It is less popular than RTS and MF and appears in \cite{Mayne1966}
under the name of Algorithm A (while MF corresponds to Algorithm B in the same paper).
M is similar to RTS, but the recursion is first run backward in time, and then forward.\\
A large body of theoretical results 
can be found in the literature on smoothing, and there are many interpretations
of estimation schemes under different perspectives, e.g. see \cite{Ljung:1976}.
However, 
clear insights on the numerical stability properties of RTS, MF and M are missing.
In order to fill this gap, we first analyze these three smoothers from an optimization and
algebraic perspective, interpreting all of them as different ways of solving 
block tridiagonal systems. Then, we exploit
this re-interpretation, together with 
the linear algebra of block tridiagonal systems, 
to obtain numerical stability properties of the smoothers. 
It turns out that RTS and MF share the same numerical stability properties, whereas M
has a stronger stability guarantee.\\

\paragraph{Outline of the paper} 
The paper proceeds as follows. In Section \ref{sec:KSblk}
we review Kalman smoothing, 
and formulate it as a least squares problem 
where the underlying system is block tridiagonal. 
In Section~\ref{sec:theory} 
we obtain bounds on the eigenvalues of these systems in terms of
the behavior of the individual blocks, and show how these bounds are related to
the stability of Kalman smoothing formulations. 
In Section~\ref{sec:forward} we show that the classic RTS smoother
implements the forward block tridiagonal (FBT) algorithm, also known as Thomas algorithm, 
for the system~\eqref{smoothingSol}. 
We demonstrate a flaw in the 
stability analysis for FBT given in~\cite{Bell2000}, 
and prove a new more powerful  stability result that shows 
that {\it any} stable system from Section~\ref{sec:theory}
can be solved using the forward algorithm. 
In Section~\ref{sec:backward}, we introduce the block backward
tridiagonal (BBT) algorithm, and prove its equivalence to M
(i.e. Algorithm A in~\cite{Mayne1966}). 
We show that M has a unique numerical stability guarantee: in contrast
to what happens with RTS, the 
eigenvalues of the individual blocks generated by M during the backward and 
forward recursions are actually
{\it independent of the condition number} of the full system. 
A numerical example illustrating these ideas is given in Section~\ref{sec:Numerics}.
Next, in Section~\ref{sec:MF}, 
we discuss the MF smoother, elucidating how it solves the block tridiagonal system, 
and showing that the two-filter formula has the same numerical stability properties of RTS
but does not have the numerical stability guarantee of M. 
Finally, in Section~\ref{sec:NewAlgo} we propose a completely new efficient 
algorithm for block tridiagonal systems.  
We conclude with a discussion of these results and their consequences.

\section{Kalman smoothing and block tridiagonal systems} 
\label{sec:KSblk}

Considering model \ref{IntroGaussModel} and using Bayes' theorem, we have 
\begin{equation}\label{Bayes}
\B{p}\left(x_k \big| z_k\right) \propto \B{p}\left(z_k\big|x_k\right)\B{p}\left(x_k\right)= \B{p}(v_k)  \B{p}(w_k)\;,
\end{equation}
and therefore the likelihood of the entire state sequence $\{x_k\}$ given the entire measurement 
sequence $\{z_k\}$ is proportional to
\begin{equation}
\label{MAP}
\begin{aligned}
\prod_{k=1}^N \B{p}(v_k)  \B{p}(w_k)
& \propto 
\prod_{k=1}^N \exp\Big( -\frac{1}{2}(z_k - H_k x_k)^\top R_k^{-1}(z_k - H_k x_k) \\
& \quad\quad-\frac{1}{2}(x_k - G_k x_{k-1})^\top Q_k^{-1}(x_k - G_k x_{k-1})\Big)\;.
\end{aligned}
\end{equation}
A better (equivalent) formulation to~\eqref{MAP} is minimizing its negative log posterior: 
\begin{equation}
\label{logMAP}
\begin{aligned}
\min_{\{x_k\}} f(\{x_k\})&:= 
\sum_{k=1}^N\frac{1}{2}(z_k - H_k x_k)^\top R_k^{-1}(z_k - H_k x_k) \\
&+\frac{1}{2}(x_k - G_k x_{k-1})^\top Q_k^{-1}(x_k - G_k x_{k-1})\;.
\end{aligned}
\end{equation}

To simplify the problem, we introduce data structures that capture the entire state sequence, measurement sequence, 
covariance matrices, and initial conditions. 

Given a sequence of column vectors $\{ u_k \}$
and matrices $ \{ T_k \}$ we use the notation
\[
\R{vec} ( \{ v_k \} )
=
\begin{bmatrix}
v_1 \\ v_2  \\ \vdots \\ v_N
\end{bmatrix}
\; , \;
\R{diag} ( \{ T_k \} )
=
\begin{bmatrix}
T_1    & 0      & \cdots & 0 \\
0      & T_2    & \ddots & \vdots \\
\vdots & \ddots & \ddots & 0 \\
0      & \cdots & 0      & T_N
\end{bmatrix} .
\]
We make the following definitions:
\begin{equation}\label{defs}
\begin{aligned}
R       & =  \R{diag} ( \{ R_k \} )
\\
Q       & =  \R{diag} ( \{ Q_k \} )
\\
H       & = \R{diag} (\{H_k\} )
\end{aligned}\quad \quad
\begin{aligned}
x       & = \R{vec} ( \{ x_k \} )
\\
\zeta      &  = \R{vec} (\{x_0, 0, \dots, 0\})
\\
z      & = \R{vec} (\{z_1,  z_2, \dots, z_N\})
\end{aligned} 
\end{equation}

\begin{equation}
\label{processG}
\begin{aligned}
G  & = \begin{bmatrix}
    \R{I}  & 0      &          &
    \\
    -G_2   & \R{I}  & \ddots   &
    \\
        & \ddots &  \ddots  & 0
    \\
        &        &   -G_N  & \R{I}
\end{bmatrix}\;.
\end{aligned}
\end{equation}

With definitions in~\eqref{defs} and~\eqref{processG}, 
problem~\eqref{logMAP} can be written
\begin{equation}\label{fullLS}
\min_{x} f(x) =  \frac{1}{2}\|Hx - z\|_{R^{-1}}^2 + \frac{1}{2}\|Gx -\zeta\|_{Q^{-1}}^2\;,
\end{equation}
where $\|a\|_M^2 = a^\top Ma$.
It is well known that finding the MAP estimate is equivalent to a least-squares problem, 
but this derivation makes the structure fully transparent. 
We now write down the linear system that needs to be solved in order 
to find the solution to~\eqref{fullLS}: 
\begin{equation}\label{smoothingSol}
(H^\top R^{-1} H + G^\top Q^{-1} G) x =  H^\top R^{-1}z + G^\top Q^{-1}\zeta\;.
\end{equation}

The linear system in~\eqref{smoothingSol} has a very special structure:
it is a symmetric positive definite block tridiagonal matrix. 
To observe it is positive definite, note that $G$ is nonsingular (for any models $G_k$) 
and $Q$ is positive definite by assumption. 
Direct computation shows that
\begin{equation}
\label{hessianApprox}
H^\top R^{-1} H + G^\top Q^{-1} G
=
\begin{bmatrix}
D_1 & A_2^\R{T} & 0 & \\
A_2 & D_2 & A_3^\R{T} & 0 \\
0 & \ddots & \ddots& \ddots & \\
& 0 & A_N & D_N
\end{bmatrix} ,
\end{equation}
with $A_k \in \mB{R}^{n\times n}$ and
$D_k \in \mB{R}^{n\times n}$ defined as follows:
\begin{equation}
\label{KalmanData}
\begin{aligned}
A_k
&=&
-Q_k^{-1}G_{k}\; , \;\\
D_k
&=&
Q_k^{-1} + G_{k+1}^\top Q^{-1}_{k+1}G_{k+1} +H_k^\top R_k^{-1}H_k\; .
\end{aligned}
\end{equation}
with $G_{N+1}^\top Q^{-1}_{N+1}G_{N+1} $ defined to be the null
$n \times n$ matrix.

This block tridiagonal structure was noted early on in~\cite{Wright1990,Fahr1991,Wright1993}.
These systems also arise in many recent extended formulations, 
see e.g.~\cite[(15)]{Bell2000},~\cite[(12)]{Bell2009},~\cite[(8.13)]{AravkinBurkePillonetto2013}.

\section{Characterizing block tridiagonal systems}
\label{sec:theory}

Consider systems of form 
\begin{equation}
\label{triSys}
E = g^{\R{T}} q^{-1} g
\end{equation}
where 
\[
q = \R{diag}\{q_1, \dots q_N\}, \quad g =  \begin{bmatrix}
    \R{I}  & 0      &          &
    \\
    g_2   & \R{I}  & \ddots   &
    \\
        & \ddots &  \ddots  & 0
    \\
        &        &   g_N  & \R{I}
\end{bmatrix}\;.
\]
Let $\lambda_{\min}$, $\lambda_{\max}$, and $\sigma_{\min}$, 
$\sigma_{\max}$ denote the minimum and maximum eigenvalues and singular
values, respectively. 
Simple upper bounds on the lower and upper 
eigenvalues of $E$ are derived in the following theorem. 
\begin{theorem}
\label{simpleBounds}

\begin{equation}
\label{lower}
\frac{\sigma^2_{\min}(g)}{\lambda_{\max}(q) } \leq 
\lambda_{\min}(E) \leq \lambda_{\max}(E) \leq 
\frac{\sigma^2_{\max}(g)}{\lambda_{\min}(q)}\;.
\end{equation}

\end{theorem}

\begin{proof}
For the upper bound, note that for any vector $v$,  
\[
v^{\R{T}} g^{\R{T}} q^{-1} g v \leq \lambda_{\max}(q^{-1}) \|gv\|^2 \leq \frac{\sigma^2_{\max}(g)}{\lambda_{\min}(q)}\|v\|^2 .
\]
Applying this inequality to a unit eigenvector for the maximum eigenvalue of $c$ gives the result.  
The lower bound is obtained analogously: 
\[
v^{\R{T}} g^{\R{T}} q^{-1} g v \geq \lambda_{\min}(q^{-1}) \|gv\|^2 \geq \frac{\sigma^2_{\min}(g)}{\lambda_{\max}(q)} \|v\|^2.
\]
Applying this inequality to a unit eigenvector for the minimum eigenvalue of $c$ completes the proof. 
\end{proof}

From this theorem, we get a simple bound on the condition number of $\kappa(B)$:
\begin{equation}
\label{cCondition}
\kappa(E) = \frac{\lambda_{\max}(E)}{\lambda_{\min}(E)} 
\leq 
\frac{\lambda_{\max}(q)\sigma^2_{\max}(g)}
{\lambda_{\min}(q) \sigma^2_{\min}(g)}\;.
\end{equation}

Since we typically have bounds on the eigenvalues of $q$, 
all that remains is to characterize the singular values of $g$
in terms of the individual $g_k$. 
This is done in the next result which uses the relation
\begin{equation}
\label{explicitForm}
g^{\R{T}}g = \left( \begin{matrix}
I + g_2^\R{T}g_2     & g_2^\R{T}  & 0       & \cdots              \\
g_2     & I + g_3^{\R{T}}a_g        &         &           \vdots    \\
\vdots  &            & \ddots  &        g_N^{\R{T}}              \\
0       &            & g_{N} & I +  g_{N+1}^{\R{T}}g_{N+1} \\
\end{matrix} \right)
\end{equation}
where we define $g_{N+1}:= 0$, so that the bottom right entry is the identity matrix. 

\begin{theorem}
\label{singularTheorem}
The following bounds hold for the singular values of $g$:
\begin{equation}
\label{bigBound}
\begin{aligned}
\max\big(0, \min_{k} &\left\{ 1 + \sigma^2_{\min}(g_{k+1}) - \sigma_{\max}(g_k) - \sigma_{\max}(g_{k+1}) \right\}\big)\\
\leq & \quad \sigma^2_{\min}(g^{\R{T}}g)\quad  \leq  \quad \sigma^2_{\max}(g^{\R{T}}g) \quad \leq \\
\max_{k}&\left\{1 + \sigma_{\max}^2(g_{k+1}) +\sigma_{\max}(g_k) + \sigma_{\max}(g_{k+1})\right\}
\end{aligned}
\end{equation}
\end{theorem}

\begin{proof}
Let $v = \R{vec}(\{v_1, \dots, v_N\})$ be any eigenvector of $g^{\R{T}}g$, so that
\begin{equation}
\label{eigenDef}
g^{\R{T}}g v = \lambda v\;.
\end{equation}
Without loss of generality, suppose that the $v_k$ component has largest norm out of $[1, \dots, N]$. 
Then from the $k$th block of~\eqref{eigenDef}, we get 
\begin{equation}
\label{eigenBlock}
g_{k}v_{k-1} + (I + g_{k+1}^{\R{T}}g_{k+1})v_k + g_{k+1}^{\R{T}}v_{k+1} = \lambda v_k\;.
\end{equation}
Let $u_k = \frac{v_k}{\|v_k\|}$. 
Multiplying~\eqref{eigenBlock} on the left by $v_k^{\R{T}}$, dividing  
by $\|v_k\|^2$, and rearranging terms, we get
\begin{equation}
\label{eigenNorm}
\begin{aligned}
1 + u_kg_{k+1}^{\R{T}}g_{k+1} u_k  - \lambda & = -u_kg_{k}\frac{v_{k-1}}{\|v_k\|} - u_kg_{k+1}^{\R{T}}\frac{v_{k+1}}{\|v_k\|}\\
& \leq \sigma_{\max} (g_k) + \sigma_{\max}(g_{k+1})\;.
\end{aligned}
\end{equation}
This relationships in \eqref{eigenNorm} yield the upper bound
\begin{equation}
\label{upperbound}
\lambda \leq 1 + \sigma_{\max}^2(g_{k+1}) +\sigma_{\max}(g_k) + \sigma_{\max}(g_{k+1})
\end{equation}
and the lower bound 
\begin{equation}
\label{lowerBound}
\begin{aligned}
\lambda &\geq 1 + u_kg_{k+1}^{\R{T}}g_{k+1} u_k - \sigma_{\max}(g_k) - \sigma_{\max}(g_{k+1})\\
& \geq 1 + \sigma^2_{\min}(g_{k+1}) - \sigma_{\max}(g_k) - \sigma_{\max}(g_{k+1})\;.
\end{aligned}
\end{equation}
Taking the minimum over $k$ in the lower bound and maximum over $k$ for the upper bound
completes the proof. The expression $\max(0, \cdots)$ in~\eqref{bigBound} arises since the singular 
values are nonnegative. 
\end{proof}

\begin{corollary}
\label{eigenvectorSpecific}
Let $v^{\min}$ be the eigenvector corresponding to $\lambda_{\min}(g^{\R{T}}g)$,
and suppose that $\|v^{\min}_k\|$ is the component with the largest norm. Then 
we have the lower bound
\begin{equation}
\label{simpleLower}
\lambda_{\min}(g^{\R{T}}g) \geq 1 + \sigma^2_{\min}(g_{k+1}) -\sigma_{\max}(g_k) - \sigma_{\max}(g_{k+1})\;.
\end{equation}
In particular, since $g_{N+1} = 0$,
\begin{equation}
\label{lastLower}
\lambda_{\min}(g^{\R{T}}g) \geq 1 -\sigma_{\max}(g_N)\;.
\end{equation}
\end{corollary}
\smallskip
%
The bound~\ref{lastLower} reveals the vulnerability of the system $g^{\R{T}}g$ 
to the behavior of the last component.

For Kalman smoothing problems, the matrix $g^{\R{T}}q^{-1}g$
corresponds to $G^{\R{T}}Q^{-1}G$.
The components $g_k$ correspond to the process models $G_k$. 
These components are often {\it identical} for all $k$, 
or they are all constructed from ODE discretizations, so that their 
singular values are similarly behaved across $k$. 
By~\eqref{simpleLower}, we see that
the last component emerges as the weakest link, since regardless of how well 
the $G_k$ are behaved for $k = 2, \dots, N-1$, the condition number
can go to infinity if any singular value for $G_N$ is larger than $1$. 
Therefore, to guard against instability, one must address instability in the final
component $G_N$.

\section{Forward Block Tridiagonal (FBT) Algorithm and the RTS smoother}
\label{sec:forward}
We now present the FBT algorithm.
Suppose for \( k = 1 , \ldots , N \),
\( b_k \in \B{R}^{n \times n} \),
\( e_k \in \B{R}^{n \times \ell} \),
\( r_k \in \B{B}^{n \times \ell} \),
and for \( k = 2 , \ldots , N \),
\( c_k \in \B{R}^{n \times n} \). 
We define the corresponding block tridiagonal system of equations
\begin{equation}
\label{BlockTridiagonalEquation}
\small
\left( \begin{matrix}
b_1     & c_2^\R{T}  & 0       & \cdots        & 0      \\
c_2     & b_2        &         &               & \vdots \\
\vdots  &            & \ddots  &               & 0         \\
0       &            & c_{N-1} & b_{N-1}       & c_N^\R{T} \\
0       & \cdots     & 0       & c_N           & b_N 
\end{matrix} \right)
\left( \begin{array}{c} 
	e_1 \\ \rule{0em}{1.5em} e_2 \\ \vdots \\ e_{N-1} \\ e_N
\end{array} \right)
=
\left( \begin{matrix} 
	r_1 \\ \rule{0em}{1.5em} r_2 \\ \vdots \\  r_{N-1} \\ r_N
\end{matrix} \right)
\end{equation}

For positive definite systems, the FBT algorithm
is defined as follows \cite[algorithm 4]{Bell2000}:
\begin{algorithm}[Forward Block Tridiagonal (FBT)]
\label{ForwardAlgorithm}
The inputs to this algorithm are 
\( \{ c_k \}_{k=2}^N \),
\( \{ b_k \}_{k=1}^N \),
and
\( \{ r_k \}_{k=1}^N \) where each
\( c_k \in \B{R}^{n \times n} \),
\( b_k \in \B{R}^{n \times n} \), and
\( r_k \in \B{R}^{n \times \ell} \).
The output is the sequence \( \{ e_k \}_{k=1}^N \) that solves equation~(\ref{BlockTridiagonalEquation}),
with each \( e_k \in \B{R}^{n \times \ell} \). 
\end{algorithm}
\begin{enumerate}

\item
\T{Set} \( d_1^f = b_1  \) and \( s_1^f = r_1 \).

\T{For} \( k = 2 \) \T{To} \( N \) \T{:}

\begin{itemize}
\item
\T{Set} \( d_k^f = b_k - c_{k} (d_{k-1}^f)^{-1} c_{k}^\R{T} \).
\item
\T{Set} \( s_k^f = r_k - c_{k} (d_{k-1}^f)^{-1} s_{k-1} \).
\end{itemize}

\item
\T{Set} \( e_N = (d_N^f)^{-1} s_N \).

\T{For} \( k = N-1 \) \T{To} \( 1 \) \T{:}
\begin{itemize}
\item
\T{Set} \( e_k = (d_k^f)^{-1} ( s_k^f - c_{k+1}^\R{T} e_{k+1} ) \).
\end{itemize}


\end{enumerate}
Before we discuss stability results for this algorithm (see theorem~\ref{algorithmStability}), 
we prove that the RTS smoother is an implementation 
of this algorithm for matrix $C$ in~\eqref{hessianApprox}.

\begin{theorem}
\label{ThomasRTS}
When applied to $C$ in~\eqref{hessianApprox} with $r= H^\top R^{-1}z + G^\top Q^{-1}\zeta$, 
Algorithm~\ref{ForwardAlgorithm} is equivalent to
 the RTS \cite{RTS} smoother. 
\end{theorem}

\begin{proof}
Looking at the very first block, we now substitute in the Kalman data structures~\eqref{KalmanData}
into step 1 of Algorithm~\ref{ForwardAlgorithm}.  
Understanding this step requires introducing some structures which may be familiar to the reader
from Kalman filtering literature. 
\begin{equation}\label{Equivalence}
\small
\begin{aligned}
P_{1|1}^{-1} & := Q_1^{-1} + H_1^\top R_1^{-1}H_1\\
P_{2|1}^{-1} & :=  (G_1P_{1|1}G_1^{\R{T}} + Q_2)^{-1} \\
&= Q_2^{-1}   - \left(Q_2^{-1}G_{2}\right)^\top\left(P_{1|1}^{-1} + 
G_{2}^\top Q^{-1}_{2}G_{2} \right )^{-1} \left(Q_2^{-1}G_{2}\right)\\
P_{2|2}^{-1} & := P_{2|1}^{-1} + H_2^\top R_2^{-1}H_2\\
d_2^f &= b_2 - c_{2}^\R{T} (d_{1}^f)^{-1} c_{2} = P_{2|2}^{-1}+ G_{3}^\top Q^{-1}_{3}G_{3}
\end{aligned}
\end{equation}
These relationships can be seen quickly from~\cite[Theorem 2.2.7]{AravkinThesis2010}. 
The matrices $P_{k|k}$, $P_{k|k-1}$ often appear in Kalman literature: they represent 
covariances of the state at time $k$ given the the measurements $\{z_1, \dots, z_k\}$, and the covariance of the 
a priori state estimate at time $k$ given measurements $\{z_1, \dots, z_{k-1}\}$, respectively. 

The key fact from~\ref{Equivalence} is that
\[
d_2^f = P_{2|2}^{-1} +G_{3}^\top Q^{-1}_{3}G_{3}\;. 
\]
Using the same computation for the generic tuple $(k,k+1)$ 
rather than $(1,2)$ establishes
\begin{equation}
\label{dDef}
d_k^f = P_{k|k}^{-1} +G_{k+1}^\top Q^{-1}_{k+1}G_{k+1}\;. 
\end{equation}
We now apply this technique to the right hand side of~\eqref{smoothingSol}, 
$r =  H^\top R^{-1}z + G^\top Q^{-1}\zeta$.  We have
\begin{equation}\label{EquivalenceRHS}
\small
\begin{aligned}
y_{2|1} & :=  \left(Q_2^{-1}G_{2}\right)^\top\left(P_{1|1}^{-1} + G_{2}^\top Q^{-1}_{2}G_{2} \right )^{-1}
\left(H_1^\top R_1^{-1}z_1 + G_1^\top P_{0|0}^{-1}x_0 \right)\\
y_{2|2} & := H_2^\top R_2^{-1}z_2 + y_{2|1} \\
s_2^f &= r_2 - c_{2}^\R{T} (d_{1}^f)^{-1} r_{1} = y_{2|2} 
\end{aligned}
\end{equation}
These relationships also follow from~\cite[Theorem 2.2.7]{AravkinThesis2010}. 
The quantities $y_{2|1}$ and $y_{2|2}$ may be familiar to the reader from the information filtering literature: 
they are preconditioned estimates 
\begin{equation}
\label{infoStructures}
\begin{aligned}
y_{k|k} &= P_{k|k}^{-1}x_{k|k}\;, \quad y_{k|k-1} & = P_{k|k-1}^{-1}x_{k|k-1}\;,
\end{aligned}
\end{equation}
where $x_{k|k}$ is the estimate of $x_k$ given $\{z_1, \dots, z_k\}$
and 
\[
x_{k|k-1} = G_k x_{k-1|k-1}
\] 
is the best prediction of the 
state $x_k$ given $\{z_1, \dots, z_{k-1}\}$.

Applying the computation to a generic index $k$, we have  $s_k^f = y_{k|k}$. 
From these results, it immediately follows that $e_N$ computed in step 2 of Algorithm~\ref{ForwardAlgorithm} 
is the Kalman filter estimate (and the RTS smoother estimate) for time point $N$ (see~\eqref{dDef}):
\begin{equation}\label{KalmanFilter}
\begin{aligned}
e_N &= (d_N^f)^{-1}s_N^f  
 = \left(P_{N|N}^{-1} + 0 \right)^{-1} P_{N|N}^{-1}x_{N|N}
 = x_{N|N}\;.
\end{aligned}
\end{equation}
\smallskip

We now establish the iteration in step 2 of Algorithm~\ref{ForwardAlgorithm}. 
First, following~\cite[(3.29)]{RTS},  
we define 
\begin{equation}
\label{Ck}
C_k = P_{k|k}G_{k+1}^{\R{T}}P_{k+1|k}^{-1}\;.
\end{equation}
\smallskip

To save space, we also use shorthand 
\begin{equation}
\label{short}
\hat P_k := P_{k|k}, \quad \hat x_{k}:= x_{k|k}\;.
\end{equation}

At the first step, we obtain
\begin{equation}\label{backwardFilter}
\begin{aligned}
e_{N-1} &= (d_{N-1}^f)^{-1}(s_{N-1}^f - c_{N}^{\R{T}}e_{N})\\
&=(\hat P_{N-1}^{-1} + G_N^{\R{T}}Q_N^{-1}G_N)^{-1}(\hat P^{-1}_{N-1}\hat x_{N-1} - G_N^{\R{T}}Q_{N}^{-1} \hat x_{N})\\
&=(\hat P_{N-1}^{-1} + G_N^{\R{T}}Q_N^{-1}G_N)^{-1}\hat P^{-1}_{N-1}\hat x_{N-1} - C_{N-1} \hat x_{N}\\
&=\hat x_{N-1} - C_{N-1}(G_n x_{N-1} - \hat x_N)\\ 
& = x_{N-1|N-1} + C_{N-1}(x_{N|N}- G_Nx_{N-1|N-1})\;,
\end{aligned}
\end{equation}
where the Woodbury inversion formula was used to get from line $3$ to line $4$. 
Comparing this to~\cite[(3.28)]{RTS}, we find that $e_{N-1} = x_{N-1|N}$, i.e. 
the RTS smoothed estimate. The computations above, when applied to the general tuple
$(k, k+1)$ instead of $(N-1, N)$, show that every $e_k$ is equivalent to $x_{k|N}$, 
which completes the proof. 
\end{proof}

In 1965, Rauch, Tung and Striebel showed that their smoother solves the maximum likelihood  
problem for $\B{p}(\{x_k\} |\{z_k\})$ \cite{RTS}, which is equivalent to~\eqref{logMAP}. 
Theorem~\ref{ThomasRTS} adds to this understanding, showing that RTS smoother
is {\it precisely} the FBT algorithm. Moreover, the proof explicitly shows how the
Kalman filter estimates $x_{k|k}$ can be obtained as the FBT
proceeds to solve system~\eqref{smoothingSol}. For efficiency, we did not
include any expressions related to the Kalman gain; the interested reader can find
these relationships in~\cite[Chapter 2]{AravkinThesis2010}. 

We now turn our attention to the stability of the forward algorithm 
for block tridiagonal systems. One such analysis appears in~\cite[Lemma 6]{Bell2000}, 
and has been used frequently to justify the Thomas algorithm as the method of choice
in many Kalman smoothing applications.  
Below, we review this result, and show that it has a critical flaw precisely for 
Kalman smoothing systems. 

\begin{lemma}\cite[Lemma 6]{Bell2000}
\label{ForwardLemma}
Suppose we are given sequences 
\( \{ q_k \}_{k=0}^N \),
\( \{ g_k \}_{k=2}^N \), and 
\( \{ u_k \}_{k=1}^N \), where 
each \( q_k \in \B{R}^{n \times n} \) is symmetric postive definite,
each \( u_k \in \B{R}^{n \times n} \) is symmetric postive semidefinite,
(in the Kalman context, $u_k$ corresponds to $H_k^T R_k^{-1}H_k$)
each \( g_k \in \B{R}^{n \times n} \).
Define 
\( b_k \in \B{R}^{n \times n} \) by
\[
b_k = u_k + q_{k}^{-1} + g_{k+1}^\R{T} q_{k+1}^{-1} g_{k+1}
\; , \; \R{where} \;
k = 1, \ldots , N 
\]
Define 
\( c_k \in \B{R}^{n \times n} \) by
\[
c_k = q_{k}^{-1} g_{k} 
\; , \; \R{where} \;
k = 2 , \ldots , N 
\]
Suppose there is an \( \alpha > 0 \) such that
all the eigenvalues of \( g_k^\R{T} q_k^{-1} g_k \) 
are greater than or equal \( \alpha \), $k=1,\dots,N$.
Suppose there is a \( \beta > 0 \) such that
all the eigenvalues of \( b_k \)
are less than or equal \( \beta \), $k=1,\dots,N$.
Further, suppose we execute
Algorithm~\ref{ForwardAlgorithm} with corresponding input sequences
\( \{ b_k \}_{k=1}^N \) and \( \{ c_k \}_{k=2}^N \).
It follows that each $d_k$ generated by the algorithm
is symmetric positive definite
and has condition number less than or equal $\beta / \alpha$. 
\end{lemma}
\smallskip

To understand what can go wrong with this analysis, 
consider $b_N$ in the Kalman smoothing context
where the corresponding matrix entries are given by 
\cite[equations (12) and (13)]{Bell2000}.
The matrix \( b_N \) in Lemma~\ref{ForwardLemma} is given by
\[
b_N = u_N + q_{N}^{-1} + g_{N+1}^\R{T} q_{N+1}^{-1} g_{N+1}
\] 
where the correspondence to \cite[(13)]{Bell2000} is given by
\( b_N \) corresponds to \( B_N \),
\( q_{N} \) corresponds to \( Q_{N} \),
\( g_N \) corresponds to \( G_N\), and
\( u_N \) corresponds to
\[
H_N^{\R{T}} R_N^{-1} H_N\;.
\]
In the context of \cite{Bell2000},
\( G_{k+1} \) is the model for the next state vector
and at \( k = N \) there is no next state vector.
Hence, \( G_{N+1} = 0 \).
Thus, in the context of Lemma~\ref{ForwardLemma}, \( a_N  = 0 \) 
and hence \( \alpha = 0 \) which contradicts the Lemma assumptions.

Does this mean that the FBT algorithm 
(and hence the Kalman filter and RTS smoother) are inherently
unstable, unless they have measurements to stabilize them? 
This has been a concern in the literature; for example, 
Bierman~\cite{Bierman1983} suggests improved stability 
as a secondary motivation for his work. 
It turns out that this concern is not justified; 
in fact we can prove a powerful theorem that relates stability of the forward block 
tridiagonal algorithm to the stability of the system~\eqref{triSys}, 
already characterized in Theorem~\ref{simpleBounds}.

\begin{theorem}
\label{algorithmStability}
Consider any block tridiagonal system $A\in \mathbb{R}^{Nn}$ 
of form~\eqref{BlockTridiagonalEquation}. 
and suppose we are given a lower bound 
$\alpha_L$ and an upper bound $\alpha_U$ on 
the eigenvalues of this system:
\begin{equation}
\label{eigenvalueBoundsIII}
0 < \alpha_L \leq \lambda_{\min}(A) \leq \lambda_{\max} (A) \leq \alpha_U\;.
\end{equation}
If we apply the FBT iteration
\[
d_k^f = b_k - c_{k} (d_{k-1}^f)^{-1} c_{k}^\R{T}, 
\]
then  
\begin{equation}
\label{eigenvalueBlockBounds}
0 < \alpha_L \leq \lambda_{\min}(d_k) \leq \lambda_{\max} (d_k) \leq \alpha_U \quad \forall k\;.
\end{equation}
\end{theorem}
In other words, the FBT iteration preserves eigenvalue bounds (and hence the condition number)
for each block, and hence will be stable when the full system is well conditioned. 

\begin{proof}
For simplicity, we will focus only on the lower bound, since the same arguments 
apply for the upper bound. 
Note that $b_1 = d_1^f$, and 
the eigenvalues of $d_1^f$ must satisfy 
\[
\alpha_L \leq \lambda_{\min}(d_1^f)
\]
since otherwise we can produce a unit-norm eigenvector $v_1 \in \mathbb{R}^{n}$ of $d_1^f$ 
with $v_1^{\R{T}} d_1^f v_1 < \alpha_L$, and then form the augmented unit vector $\widetilde v_1 \in \mathbb{R}^N$
with $v_1$ in the first block, and every other entry $0$. Then we have 
\[
\widetilde v_1^{\R{T}} A \widetilde v_1 < \alpha L\;,
\]
which violates~\eqref{eigenvalueBounds}. 
Next, note that 
\begin{equation}
\label{OneStep}
\small
S_1
A
S_1^{\R{T}}
=
\left( \begin{matrix}
b_1     & 0  & 0       & \cdots        & 0      \\
0     & d_2^f        &         &               & \vdots \\
\vdots  &            & \ddots  &               & 0         \\
0       &            & c_{N-1} & b_{N-1}       & c_N^\R{T} \\
0       & \cdots     & 0       & c_N           & b_N 
\end{matrix} \right)
\end{equation}
where 
\[
d_2^f = b_2 - c_2(d_1^f)^{-1}c_2^{\R{T}}
\]
and
\[
S_1 = 
\left( \begin{matrix}
I     & 0  & 0       & \cdots        & 0      \\
-c_2(d_1^f)^{-1}     & I        &         &               & \vdots \\
\vdots  &            & \ddots  &               & 0         \\
0       &            & 0  & I      & 0 \\
0       & \cdots     & 0       & 0          & I
\end{matrix} \right)\;.
\]
Suppose now that $d_2^f$ has an eigenvalue that is less than $\alpha_L$. 
Then we can produce a unit eigenvector $v_2$ of $d_2^f$ with $v_2^{\R{T}}d_2^f v_2 < \alpha_L$, 
and create an augmented unit vector 
\[
\widetilde v_2 = \begin{bmatrix} 0_{1\times n} & v_2^{\R{T}} & 0_{1 \times n(N-2)}\end{bmatrix}^{\R{T}}
\]
which satisfies 
\begin{equation}
\label{lowerBreak}
\widetilde v_2^{\R{T}} S_1 AS_1^{\R{T}} \widetilde v_2 < \alpha_L\;.
\end{equation}
Next, note that 
\[
\hat v_2^{\R{T}} :=  \widetilde v_2^{\R{T}} S_1 = \begin{bmatrix} -v_2^{\R{T}}c_2(d_1^f)^{-1} 
& v_2^{\R{T}} & 0_{1 \times n(N-2)}\end{bmatrix}^{\R{T}}\;,
\]
so in particular $\|\hat v_2\| \geq 1$. From~\eqref{lowerBreak}, we now have
\[
\hat v_2^T A \hat v_2 < \alpha_L \leq \alpha_L \|v_2\|^2\;,
\]
which violates~\eqref{eigenvalueBoundsIII}. To complete the proof, note that the lower $n(N-1) \times n(N-1)$
block of $S_1 AS_1^{\R{T}}$ is identical to that of $A$, with~\eqref{eigenvalueBoundsIII} holding for this modified system. 
The reduction technique can now be repeatedly applied. 
\end{proof}

Note that Theorem~\ref{algorithmStability} applies to {\it any} block tridiagonal system
satisfying~\eqref{eigenvalueBoundsIII}. 
When applied to the Kalman smoothing setting, if the system $G^{\R{T}}Q^{-1}G$ is well-conditioned, 
we know that the FBT (and hence the Kalman filter and RTS smoother)
will behave well for any measurement models. Recall that a lower bound for the condition number
of the full system in terms of the behavior of the blocks in Theorem~\ref{singularTheorem} 
and Corollary~\eqref{eigenvectorSpecific}. 
Moreover, even if $G^{\R{T}}Q^{-1}G$ has a bad 
condition number, it is possible that the measurement term $H^{\R{T}}R^{-1}H$ (see~\eqref{smoothingSol})
will improve the condition number. More general Kalman smoothing applications may not have this advantage. 
For example, the initialization procedure in~\cite{AravkinIeee2011} requires the inversion of systems 
analogous to $G^{\R{T}}Q^{-1}G$, without a measurement term.

\subsection{Invertible Measurement Component}

We now return to the system~\eqref{smoothingSol}, and briefly 
consider the case where $H^TR^{-1}H$ is an invertible matrix. 
Note that this is not true in general, and in fact our 
stability analysis, as applied to the Kalman smoothing problem, 
did not use any assumptions on this term. 

However, if we know that 
\begin{equation}
\label{introF}
\Lambda:= H^TR^{-1}H
\end{equation}
is an invertible matrix, then we can consider an alternative approach to 
solving~\eqref{smoothingSol}. Applying the Woodbury inversion formula, 
we obtain
\begin{equation}
\label{WoodMeas}
(G^TQ^{-1}G + \Lambda)^{-1}
=
\Lambda^{-1} - \Lambda^{-1} G^T(Q + G\Lambda^{-1}G^T)^{-1}G\Lambda^{-1}
\end{equation}
Now, the solution to~\eqref{smoothingSol} can be found 
by applying this explicit inverse to the right hand side
\[
\R{rhs} := H^TR^{-1}z + G^TQ^{-1}\zeta
\] 
and the key computation becomes 
\begin{equation}
\label{WoodburyInverse}
(Q + G\Lambda^{-1}G^T)x = G\Lambda^{-1}\R{rhs}\;.
\end{equation}

Note that the matrix $Q+G\Lambda^{-1}G^T$ is block tridiagonal, 
since $Q$ and $\Lambda^{-1}$ are block diagonal, and $G$ is lower block bidiagonal. 
Therefore, we have reduced the problem to a system of the form~\eqref{BlockTridiagonalEquation}. 
Moreover, at a glance we can see the lower eigenvalues of this system are bounded below by 
the eigenvalues of $Q$, while upper bounds can be constructed from eigenvalues of $Q$, 
 $G_k$ and $\Lambda^{-1}$. Under very mild conditions, this system can be solved 
 in a stable way by the forward tridiagonal algorithm, which also give a modified filter and smoother. 
 This is not surprising, since we have assumed the extra hypothesis that $\Lambda$ is invertible. 

\section{Backward Block Tridiagonal Algorithm and the M smoother}
\label{sec:backward}

Having abstracted Kalman smoothing problems and algorithms 
to solutions of block tridiagonal systems, it is natural to consider 
alternative algorithms in the context of these systems. 
In this section, we propose a new scheme, namely {\it backward} block tridiagonal (BBT)
algorithm, and show it is equivalent to the M smoother \cite{Mayne1966}
when applied to the Kalman smoothing setting. 
We also prove a stability result for ill-conditioned
block tridiagonal systems. 

Let us again begin with system \eqref{BlockTridiagonalEquation}. 
If we substract \( c_N^\R{T} b_N^{-1} \) times row \( N \) from row \( N - 1 \),
we obtain the following equivalent system:
\[
\small
\begin{array}{lll}
&\left( \begin{matrix}
b_1     & c_2^\R{T}  & 0       & \cdots                            & 0      \\
c_2     & b_2        &         &                                   & \vdots \\
\vdots  &            & \ddots  &                                   & 0      \\
0       &            & b_{N-2} & c_{N-1}^\R{T}                     & 0      \\
0       &            & c_{N-1} & b_{N-1} - c_N^\R{T} b_N^{-1} c_N  & 0      \\
0       & \cdots     & 0       & c_N                               & b_N 
\end{matrix} \right)
\left( \begin{matrix} 
	e_1 \\ e_2 \\ \vdots \\ e_{N-1} \\ e_N
\end{matrix} \right)
\\
=
&
\left( \begin{matrix} 
	r_1 \\ r_2 \\ \vdots \\  r_{N-1} - c_N^\R{T} b_N^{-1} r_N \\ r_N
\end{matrix} \right)\;.
\end{array}
\]
We iterate this procedure until we reach the first row of the matrix,
using \( d_k \) to denote the resulting diagonal blocks,
and \( s_k \) the corresponding right hand side of the equations; i.e.,
\[
\begin{aligned}
d_N^b &=  b_N\;,\;
d_k^b  =  b_k - c_{k+1}^\R{T} (d_{k+1}^b)^{-1} c_{k+1} \quad (k=N-1 , \cdots , 1)
\\
s_N^b &=  e_N \;,\;
s_k^b  = r_k - c_{k+1}^\R{T} (d_{k+1}^b)^{-1} s_{k+1} \quad (k=N-1 , \cdots , 1)\;.
\end{aligned}
\]
We obtain the following lower triangular system:
\begin{equation}
\small
\left( \begin{matrix}
d_1^b     & 0          & \cdots    &         & 0          \\
c_2     & d_2^b        & 0         & \dots   & 0           \\
        &            & \ddots    &         & \vdots      \\
\vdots  &            & c_{N-1}   & d_{N-1}^b & 0           \\
0       &            &           & c_N     & d_N^b 
\end{matrix} \right)
\left( \begin{matrix} 
	e_1 \\ e_2 \\ \vdots \\ e_{N-1} \\ e_N 
\end{matrix} \right)
=
\left( \begin{matrix} 
	s_1 \\ s_2 \\ \vdots \\ s_{N-1} \\ r_N 
\end{matrix} \right)
\end{equation}
Now we can solve for the {\it first} block vector
and then proceed back down, doing back substitution. 
We thus obtain the following algorithm:

\begin{algorithm}[Backward block tridiagonal (BBT)]
\label{algoBBT}
The inputs to this algorithm are 
\( \{ c_k \} \),
\( \{ b_k \} \),
and
\( \{ r_k \} \).
The output is a sequence \( \{ e_k \} \)
that solves equation~(\ref{BlockTridiagonalEquation}).
\end{algorithm}
\begin{enumerate}

\item
\T{Set} \( d_N^b = b_N  \) and \( s_N^b = r_N \).

\T{For} \( k = N-1, \ldots , 1 \),
\begin{itemize}
\item
\T{Set} \( d_k^b = b_k - c_{k+1}^\R{T} (d_{k+1}^b)^{-1} c_{k+1} \).
\item
\T{Set}\( s_k^b = r_k - c_{k+1}^\R{T} (d_{k+1}^b)^{-1} s_{k+1} \).
\end{itemize}
\item
\T{Set} \( e_1 = (d_1^b)^{-1} s_1^b \).

\T{For} \( k = 2 , \ldots , N \),
\begin{itemize}
\item
\T{Set} \( e_k = (d_k^b)^{-1} ( s_k^b - c_k e_{k-1} ) \).
\end{itemize}

\end{enumerate}

Before we discuss stability results, we show that this algorithm
is equivalent to the M smoother~\cite{Mayne1966}.
The recursion in~\cite[ Algorithm A]{Mayne1966}, translated to our notation, is 
\begin{eqnarray}
\label{MayneP}P_k &=& G_{k+1}^T\left[ I - P_{k+1} C_{k+1} \Delta_{k+1} C_{k+1}^T\right]P_{k+1}G_{k+1}\\
\nonumber&& + H_k^TR_k^{-1}H_k \\
\label{MayneDel}\Delta_k &=& \left[I + C_k^T P_{k+1}C_k \right]^{-1}\\
\label{MayneQ}\phi_k &=& -H_k^TR_k^{-1}z_k + G_{k+1}^T\left[I - P_{k+1}C_k\Delta_kC_k^T\right]\phi_{k+1}\;,
\end{eqnarray} 
where $Q_k = C_kC_k^T$. The recursion is initialized as follows: 
\begin{equation}
P_N = H_N^TR_N^{-1}H_N\;, \quad \phi_N = -H_N^TR_N^{-1}z_n\;.
\end{equation}

Before stating a theorem, we prove a useful linear algebraic result. 
\begin{lemma} 
Let $P$ and $Q$ be any invertible matrices. Then 
\begin{equation}
\label{PQlemma}
P - P(Q^{-1} + P)^{-1}P = Q^{-1} - Q^{-1}(Q^{-1} + P)^{-1}Q^{-1}\;.
\end{equation}
\end{lemma}

\begin{proof}
Starting with the left hand side, write $P = P + Q^{-1} - Q^{-1}$. Then we have 
\[
\begin{aligned}
P - P(Q^{-1} + P)^{-1}P 
& = 
P - P(Q^{-1} + P)^{-1}(P + Q^{-1} - Q^{-1}) \\
&= 
P(Q^{-1} + P)^{-1}Q^{-1} \\
&=
(P+ Q^{-1} - Q^{-1})(Q^{-1} + P)^{-1}Q^{-1} \\
&= 
Q^{-1} - Q^{-1}(Q^{-1} + P)^{-1}Q^{-1}
\end{aligned}
\]
\end{proof}

\begin{theorem}
\label{BBTthm}
When applied to $C$ in~\eqref{hessianApprox} with $r= H^\top R^{-1}z + G^\top Q^{-1}\zeta$, BBT is equivalent to
the M smoother, i.e. \cite[ Algorithm A]{Mayne1966}. In particular, $d_k^b$ in Algorithm~\ref{algoBBT}
corresponds to $P_k + Q_k^{-1}$, while $s_k^b$ in Algorithm~\ref{algoBBT} is precisely $-\phi_k$ in 
recursion~\eqref{MayneP}---\eqref{MayneQ}.   
\end{theorem}

\begin{proof}
Using $c_k$ and $b_k$ in~\eqref{KalmanData}, the relationships above are immediately seen to hold
for step $N$. As in the other proofs, we show only the next step. 
From~\eqref{MayneP}, we have 
\begin{equation}
\begin{aligned}
P_{N-1} &= H_{N-1}^TR_{N_1}^{-1}H_{N} + G_{N}^T \Phi G_N\\
\Phi & = P_N - P_N(C_{N}\Delta_{N}C_{N}^T)P_N \\
& = P_N - P_N(Q_{N} - Q_{N}(P_{N}^{-1} + Q_{N}^{-1})^{-1}Q_{N})P_N\\
&= P_N - P_N(Q_{N}^{-1} + P_N)^{-1}P_N\\
& = Q_{N}^{-1} - Q_N^{-1} (d_N^b)^{-1} Q_N^{-1}\\
\end{aligned}
\end{equation}
where the Woodbury inversion formula was used twice to get from line 2 to line 4, and 
Lemma~\ref{PQlemma} together with the definition of $d_N^b$ was used to get from line 4 to line 5. 

Therefore, we immediately have 
\[
\begin{aligned}
P_{N-1} &= H_{N-1}^TR_{N_1}^{-1}H_{N} + G_{N}^T (Q_{N}^{-1} - Q_N^{-1} d_N^{-1} Q_N^{-1}) G_N\\
& = d_{N-1}^b - Q_{N-1}
\end{aligned}
\]
as claimed. 
Next, we have 
\begin{equation}
\begin{aligned}
\phi_{N-1} &= -H_{N-1}^TR_{N-1}^{-1}z_N + G_N^T(I - P_N(C_{N}\Delta_{N}C_{N}^T))q_N\\
&= -H_{N-1}^TR_{N-1}^{-1}z_N - G_N^T(I - P_N(Q_N^{-1} + P_N)^{-1})s_N\\
&= -H_{N-1}^TR_{N-1}^{-1}z_N - G_N^T(P_N^{-1} + Q_N)^{-1}P_N^{-1})s_N\\
&= -H_{N-1}^TR_{N-1}^{-1}z_N - G_N^T(Q_N^{-1}(P_N + Q_N^{-1})^{-1})s_N\\
& = -s_{N-1}^b\;.
\end{aligned}
\end{equation}
Finally, note that the smoothed estimate give in~\cite[(A.8)]{Mayne1966} (translated to our notation)
\[
\hat x_1 = -(P_1 + Q_1^{-1})^{-1}(-s_1^b - Q_1^{-1}x_0)
\]
is precisely $(d_1^b)^{-1}r_1$, which is the estimate $e_1$ in step 2 of~Algorithm~\ref{algoBBT}.
The reader can check that the forward recursion in~\cite[(A.9)]{Mayne1966} is equivalent to the recursion 
in step 2 of~Algorithm~\ref{algoBBT}.
\end{proof}

Next, we show
that the BBT algorithm 
has the same stability result as the FBT algorithm. 
\begin{theorem}
\label{algorithmStabilityBackward}
Consider any block tridiagonal system $A\in \mathbb{R}^{Nn}$ 
of form~\eqref{BlockTridiagonalEquation}. 
and suppose we are given the bounds 
$\alpha_L$ and $\alpha_U$ for the lower and 
upper bounds of the eigenvalues of this system
\begin{equation}
\label{eigenvalueBounds}
0 < \alpha_L \leq \lambda_{\min}(A) \leq \lambda_{\max} (A) \leq \alpha_U\;.
\end{equation}
If we apply the BBT iteration
\[
d_k^b = b_k - c_{k+1}^\R{T} (d_{k+1}^b)^{-1} c_{k+1} 
\]
then we have 
\begin{equation}
\label{eigenvalueBlockBounds}
0 < \alpha_L \leq \lambda_{\min}(d_k) \leq \lambda_{\max} (d_k) \leq \alpha_U \quad \forall k\;.
\end{equation}
\end{theorem}
\begin{proof}
Note first that $d_N^b = b_N$, and satisfies~\eqref{eigenvalueBlockBounds}
by the same argument as in the proof of Theorem~\ref{algorithmStabilityBackward}.
Define  
\[
S_N = 
\left( \begin{matrix}
I     & 0  & 0       & \cdots        & 0      \\
0     & I        &         &               & \vdots \\
\vdots  &            & \ddots  &               & 0         \\
0       &            & 0  & I      & 0 \\
0       & \cdots     & 0       & -c_N(d_{N-1}^b)^{-1}         & I
\end{matrix} \right)\;.
\]
and note that 
\[
S_N^{\R{T}}AS_N 
=
\left( \begin{matrix}
d_1^b     & c_2^{\R{T}}          & \cdots    &         & 0          \\
c_2     & d_2^b        & c_3^{\R{T}}         & \dots   & 0           \\
0        &            & \ddots    &     c_{N-1}^{\R{T}}    & \vdots      \\
\vdots  &            & c_{N-1}   & d_{N-1}^b & 0           \\
0       &       \cdots     &           & 0     & d_N^b 
\end{matrix} \right)
\]
Now an analogous proof to that of Theorem~\ref{algorithmStability}
can be applied to show the upper $n(N-1)\times n(N-1)$ block of 
$S_N^{\R{T}}AS_N$ satisfies~\eqref{eigenvalueBlockBounds}. 
Applying this reduction iteratively completes the proof. 
\end{proof}

Theorems~\ref{algorithmStability} and~\ref{algorithmStabilityBackward}
show that both forward and backward tridiagonal algorithms are stable 
when the block tridiagonal systems they are applied to are well conditioned.
For a lower bound on the condition number in the Kalman smoothing
context, see Theorem~\ref{singularTheorem} 
and Corollary~\eqref{eigenvectorSpecific}. 

However, a different analysis can also be done for a particular
class of block tridiagonal systems. This result, which applies to 
Kalman smoothing systems, shows that {\it the backward 
algorithm can behave stably even when the tridiagonal 
system has null singular values}. In other words,
the eigenvalues of the individual blocks 
generated by M during the backward and forward recursions 
are actually independent of the condition number of the full system, 
which is a stronger result than we have for RTS.

For \( v \in \B{R}^{n \times n} \)
we use the notation \( | v | \)
for the operator norm of the matrix \( v \); i.e.,
\[
| v | = \sup \{ | v w | \; : \; w \in \B{R}^n \; , \; |w| = 1 \}
\]

\begin{theorem}
\label{InductionRelationTheorem}
Suppose that the matrices \( c_k \) and \( b_k \) are given by
\begin{eqnarray*}
c_k & = & - q_k^{-1} g_k
\\
b_k & = & q_k^{-1} + u_k + g_{k+1}^\R{T} q_{k+1}^{-1} g_{k+1}
\end{eqnarray*}
where 
each \( q_k \in \B{R}^{n \times n} \) is positive definite,
each \( u_k \in \B{R}^{n \times n} \) is positive semi-definite
and \( g_{N+1} = 0 \).
It follows that \( d_k - q_k^{-1} \) is positive semi-definite 
for all \( k \).
Furthermore, if \( \alpha \) is a bound for 
\( | q_k | \), \( | q_k^{-1} | \), \( | u_k | \), and \( | g_k | \),
Then the condition number of \( d_k^b \) is bounded by
\( \alpha^2 + \alpha^6 \).
\end{theorem}

\begin{proof}
We note that
\( d_N^b = q_N^{-1} + u_N \) so this conditions bound holds for \( k = N \).
Furthermore \( d_N^b - q_N^{-1} = u_N \), so positive semi-definite
assertion holds for \( k = N \).

We now complete the proof by induction; i.e.,
suppose \( d_{k+1}^b - q_{k+1}^{-1} \) is positive semi-definite
\begin{eqnarray*}
d_k^b 
& = & 
b_k - c_{k+1}^\R{T} (d_{k+1}^b)^{-1} c_{k+1}
\\
d_k^b - q_k^{-1}
& = &
u_k + g_{k+1}^\R{T} q_{k+1}^{-1} g_{k+1} 
- g_{k+1}^\R{T}  q_{k+1}^{-1}   (d_{k+1}^b)^{-1} q_{k+1}^{-1} g_{k+1}
\\
& = &
u_k + g_{k+1}^\R{T} q_{k+1}^{-1} 
	\left[ q_{k+1} -  (d_{k+1}^b)^{-1} \right] q_{k+1}^{-1} g_{k+1}
\end{eqnarray*}
The assumption that \( d_{k+1}^b - q_{k+1}^{-1} \) is positive semi-definite
implies that that \( q_{k+1} -  (d_{k+1}^b)^{-1} \) is positive semi-definite.
It now follows that
\( d_k^b - q_k^{-1} \) is the sum of positive semi-definite matrices
and hence is positive semi-definite, which completes the induction
and hence proves that \( d_k^b - q_k^{-1} \) is positive semi-definite.

We now complete the proof by showing that the condition number bound holds
for index \( k \).
Using the last equation above, we have
\begin{eqnarray*}
| d_k^b | 
& \leq & 
| u_k | + | g_{k+1}^\R{T} |^2 | q_{k+1}^{-1} |^2 |  q_{k+1} |
\\
& \leq &
\alpha + \alpha^5
\end{eqnarray*}
Hence the maximum eigenvalue of \( d_k^b \) is less than or equal
\( \alpha + \alpha^5 \).
In addition, since \( d_k^b - q_k^{-1} \) is positive semi-definite,
the minimum eigenvalue of \( d_k^b \) is greater than or equal
the minimum eigenvalue of \( q_k^{-1} \),
which is equal to the reciprocal of the maximum eigenvalue of \( q_k \).
Thus the minimum eigenvalue of \( d_k^b \) is greater than or equal
\( 1 / \alpha \).
Thus the condition number of \( d_k^b \) is bounded by
\( \alpha^2 + \alpha^6 \)
which completes the proof.
\end{proof}

\section{Numerical Example}
\label{sec:Numerics}

To put all of these ideas in perspective, we consider a toy numerical example. 
Let $n = 1$ and $N = 3$, let $q_k = 1$ for all $k$, and let $a_k = 120$. 
Then the system $a^{\R{T}}q^{-1}a$ in~\eqref{triSys} is given by 
\[
\begin{bmatrix}
14401 & 120 & 0\\
120 & 14401 & 120 \\
0 & 120 & 1
\end{bmatrix}\;,
\]
and its minimum eigenvalue is $4.8 \times 10^{-9}$.
To understand what goes wrong, we first note that the minimum and
maximum eigenvalues of all the $g_k$'s except the last one coincide 
in this case, so the general condition of Corollary~\ref{eigenvectorSpecific}
will hold everywhere except at the last coordinate. 

Now that we suspect the last coordinate, we can check the eigenvector corresponding
to the minimum eigenvalue: 
\[
v^{\min} = \begin{bmatrix} 0.001 & -.008 & 1\end{bmatrix}^{\R{T}}\;.
\]
Indeed, the component of the eigenvector with the largest norm 
occurs precisely in the last component, so we are in the case described by 
Corollary~\ref{eigenvectorSpecific}.
In order to stabilize the system under a numerical viewpoint, we have an option to make the $(3,2)$ 
and $(2,3)$ coordinates less than $1$ in absolute value; let's take them to be 
$0.9$ instead of 120. The new system has lowest eigenvalue $1$. 
If we expand the system (by increasing the size of $\{g_k\}$, $\{q_k\}$), 
we find this stabilization technique works regardless of the size, since the last
component is always the weakest link when all $g_k$ are equal. 

Next, suppose we apply FBT to the original (unstable)
system. We would get blocks
\begin{equation}
\label{d3f}
\begin{aligned}
d_1^f = 14401, \quad d_2^f = 14400, \quad d_3^f = 4.8222 \times 10^{-9}\;.
\end{aligned}
\end{equation}
This toy example shows Theorem~\ref{algorithmStability} in action ---
indeed, the eigenvalues of the blocks are bounded above and below by 
the eigenvalues of the full system. Unfortunately, in this case this leads to a terrible 
condition number, since we are working with an ill-conditioned system. 

Now, suppose we apply the BBT. We will get
the following blocks: 
\[
d_3^b = 1, \quad d_2^b = 1, \quad d_1^b = 1\;.
\]
Note that even though the {\it blocks} are stable, this does not mean 
that the backward algorithm can accurately solve the system $g^{\R{T}}q^{-1}g x = b$. 
In practice, it may perform better or worse than the 
forward algorithm for ill-conditioned systems, depending on the problem. 
Nonetheless, the condition of the blocks does not depend on the condition number of the full system
as in the forward algorithm, as shown by Theorem~\ref{algorithmStabilityBackward} 
and this example. 

\section{Two filter block tridiagonal algorithm and the MF smoother}
\label{sec:MF}

In this section, we explore the MF smoother, and show how it solves 
the block tridiagonal system~\eqref{BlockTridiagonalEquation}. 
Studying the smoother in this way allows us to also characterize 
its stability using theoretical results developed in this paper. 

Consider a linear system of the form~\eqref{BlockTridiagonalEquation}.
As in the previous sections, let
$d_k^f, s_k^f$ denote the forward matrix and vector terms 
obtained after step 1 of Algorithm~\ref{ForwardAlgorithm}, 
and $d_k^b, s_k^b$ denote the terms obtained after step 1 
of Algorithm~\ref{algoBBT}. Finally, recall that 
$b_k, r_k$ refer to the diagonal terms and right hand side 
of system~\eqref{BlockTridiagonalEquation}. 

The following algorithm uses elements of 
both forward and backward algorithms. We dub it 
the Block Diagonalizer. 

\begin{algorithm}[Two filter block tridiagonal algorithm]
\label{algoDiag}
The inputs to this algorithm are 
\( \{ c_k \} \),
\( \{ b_k \} \),
and
\( \{ r_k \} \).
The output is a sequence \( \{ e_k \} \)
that solves equation~(\ref{BlockTridiagonalEquation}).
\end{algorithm}
\begin{enumerate}

\item
\T{Set} \( d_1^f = b_1  \). \T{Set} \( s_1^f = r_1 \).

\T{For} \( k = 2 \) \T{To} \( N \) \T{:}

\begin{itemize}
\item
\T{Set} \( d_k^f = b_k - c_{k} (d_{k-1}^f)^{-1} c_{k}^\R{T} \).
\item
\T{Set} \( s_k^f = r_k - c_{k} (d_{k-1}^f)^{-1} s_{k-1} \).
\end{itemize}

\item
\T{Set} \( d_N^b = b_N  \) and \( s_N^b = r_N \).

\T{For} \( k = N-1, \ldots , 1 \),
\begin{itemize}
\item
\T{Set} 
\( d_k^b = b_k - c_{k+1}^\R{T} (d_{k+1}^b)^{-1} c_{k+1} \).
\item
\T{Set}
\( s_k^b = r_k - c_{k+1}^\R{T} (d_{k+1}^b)^{-1} s_{k+1} \).
\end{itemize}

\item
For \( k = 1 , \ldots , N \),
set \( e_k = (d_k^f + d_k^b - b_k)^{-1} ( s_k^f + s_k^b - r_k) \).
\end{enumerate}

It is easy to see why the MF smoother is often used in practice. 
Steps 1 and 2 can be done in parallel on two processors. 
Step 3 is independent across $k$, and so can be done
in parallel with up to $N$ processors. In the next section, 
we will use the insights into block tridiagonal systems 
to propose a new and more efficient algorithm. 
First, we finish our analysis of the MF smoother. 

Let $C$ denote the linear system in~\eqref{BlockTridiagonalEquation}, 
and $F$ denote the matrix whose action is equivalent 
to step 1 of Algorithm~\ref{algoDiag}, 
so that $FC$ is upper block triangular, and $Fr$ recovers blocks $\{s_k^f\}$.  
Let $B$ denote the matrix whose action is equivalent 
to steps 2 of Algorithm~\ref{algoDiag}, 
so that $BC$ is lower block triangular, and $Br$ recovers blocks $\{s_k^b\}$.  

\begin{theorem}
\label{thm:FullMF}
The solution $e$ to~\eqref{BlockTridiagonalEquation} is given by
Algorithm~\ref{algoDiag}. 
\end{theorem}

\begin{proof}
The solution $e$ returned by Algorithm~\ref{algoDiag} can be written as follows: 
\begin{equation}
\label{FBsol}
e = ((F+B -I)C)^{-1}(F+B-I)r\;.
\end{equation}
To see this, note that $FC$ has the same blocks above the diagonal as $C$, 
and zero blocks below the diagonal. Analogously, $BC$ has the 
same blocks above the diagonal as $C$. Then 
\[
FC + BC - C 
\]
is block diagonal, with diagonal blocks given by 
$d_k^f + d_k^b - b_k$.
Finally, applying the system $F+B-I$ to $r$ 
yields the blocks $s_k^f + s_k^b - r_k$. 

The fact that $e$ solves~\eqref{BlockTridiagonalEquation} follows from the following calculation:
\[
\begin{aligned}
Ce &= C ((F+B -I)C)^{-1}(F+B-I)r \\
&= C C^{-1}(F+B -I)^{-1}(F+B-I)r \\
& = r\;.
\end{aligned}
\]
\end{proof}
Note that applying $F$ and $B$ to construct
$(F+B-I)r$ is an $O(n^3N)$ operation, since both $F$ and $B$ 
require recursive inversions of $N$ blocks, each of size $n\times n$, 
as is clear from steps 1,2 of Algorithm~\ref{algoDiag}. 

Before we consider stability conditions (i.e. conditions that guarantee 
invertibility of $d_k^f + d_k^b - b_k$), we demonstrate the equivalence 
of Algorithm~\ref{algoDiag} to the Mayne-Fraser smoother. 

\begin{lemma}
\label{MFchar}
When applied to the Kalman smoothing system~\eqref{smoothingSol}, 
the Mayne-Fraser smoother is equivalent to Algorithm~\ref{algoDiag}. 
In particular, the MF update can be written  
\begin{equation}
\label{MFsolution}
\hat x_k = (d_k^f + d_k^b - b_k)^{-1} (s_k^f + s_k^b - r_k)\;.
\end{equation}
\end{lemma}

\begin{proof}
The MF smoother solution is given in~\cite[(B.9)]{Mayne1966}:
\[
\hat x_k = -(P_k + F_k)^{-1} (q_k + g_k)\;
\]
From Theorem~\ref{BBTthm}, we know $q_k = -s_k^b$, and 
$P_k = d_k^b - Q_{k}^{-1}$. 
Next, $F_k = \sigma_{k|k-1}^{-1}$ from~\cite[(B.6)]{Mayne1966}, 
and we have 
\[
\begin{aligned}
F_k &= \sigma_{k|k-1}^{-1} = \sigma_{k|k}^{-1} - H_k^TR_k^{-1}H_k\quad \text{~\cite[Chapter 2]{AravkinThesis2010} } \\
& = d_k^f - G_{k+1}^TQ_{k+1}^{-1}G_{k+1} - H_k^TR_k^{-1}H_k \quad\text{by~\eqref{dDef}} 
\end{aligned}
\]
Therefore, 
\[
\begin{aligned}
P_k + F_k &= d_k^b + d_k^f- Q_{k}^{-1} - G_{k+1}^TQ_{k+1}^{-1}G_{k+1} - H_k^TR_k^{-1}H_k  \\
&= d_k^b + d_k^f - b_k \quad\text{by~\eqref{KalmanData}}\;.
\end{aligned}
\]
Finally, $g_k = -\sigma_{k|k-1}^{-1}x_{k|k-1}$ from~\cite[(B.7)]{Mayne1966} and we have 
\[
\begin{aligned}
g_k = -\sigma_{k|k-1}^{-1}x_{k|k-1} &=  y_{k|k-1}\quad\text{by~\eqref{infoStructures}} \\
& = -(s_k^f + H_k^TR_k^{-1}z_k)\quad\text{by~\eqref{EquivalenceRHS}}\\
& = -(s_k^f - r_k)\quad\text{by~\eqref{smoothingSol}}\;.
\end{aligned}
\]
This gives $-(q_k + g_k) = s_k^f + s_k^b - r_k$, and the lemma is proved. 
\end{proof}

Lemma~\ref{MFchar} characterizes MF through its
relationship to the tridiagonal system~\eqref{BlockTridiagonalEquation}.
Our final result is with stability of the MF scheme, which can be understood
by considering the solution~\eqref{MFsolution}.

Note that at $k=1$, we have $d_1^f = b_1$, and so the 
matrix to invert in~\eqref{MFsolution} is simply $d_1^b$. At $k = N$, we have $d_N^b = b_N$, 
and the matrix to invert is $d_N^f$. Therefore, the MF scheme is vulnerable 
to numerical instability (at least at time $N$) when the system~\eqref{BlockTridiagonalEquation}
is ill-conditioned.  In particular, if we apply MF to the numerical example in 
Section~\ref{sec:Numerics}, block $d_3^f$~\eqref{d3f} will have to be used.  

The following theorem shows that the MF soother has the same 
stability guarantees as the RTS smoother for well-conditioned systems.

\begin{theorem}
\label{thm:MFstability}
Consider any block tridiagonal system $A\in \mathbb{R}^{Nn}$ 
of form~\eqref{BlockTridiagonalEquation}. 
and suppose we are given the bounds 
$\alpha_L$ and $\alpha_U$ for the lower and 
upper bounds of the eigenvalues of this system
\begin{equation}
\label{eigenvalueBounds}
0 < \alpha_L \leq \lambda_{\min}(A) \leq \lambda_{\max} (A) \leq \alpha_U\;.
\end{equation}
Then we also have 
\begin{equation}
\label{eigenvalueBlockBounds}
0 < \alpha_L \leq \lambda_{\min}(d_k^f + d_k^b - b_k) \leq \lambda_{\max} (d_k^f + d_k^b - b_k) \leq \alpha_U \quad \forall k\;.
\end{equation}
\end{theorem}

\begin{proof}
At every intermediate step, it is easy to see that 
\[
\begin{aligned}
d_k^f + d_k^b - b_k &= b_k - c_{k}(d_{k-1}^f)^{-1} c_{k}^T + b_k - c_{k+1}^T(d_{k+1}^b)^{-1}c_{k+1} - b_k\\
& = b_k - c_{k}(d_{k-1}^f)^{-1} c_{k}^T  - c_{k+1}^T(d_{k+1}^b)^{-1}c_{k+1}
\end{aligned}
\]
This corresponds exactly to isolating the middle block of the three by three system 
\[
\left(
\begin{matrix}
d_{k-1}^f & c_k^T & 0\\
c_k & b_k & c_{k+1}^T \\
0 & c_{k+1} & d_{k+1}^b
\end{matrix}
\right)\;.
\]
By Theorems~\ref{algorithmStability} and~\ref{algorithmStabilityBackward}, 
the eigenvalues of this system are bounded by
the be eigenvalues of the full system. Applying these theorems to the middle block
shows that the system in~\eqref{MFsolution} also satisfies such a bound. 
\end{proof}
Therefore, for well-conditioned systems, the MF scheme 
shares the stability results of RTS. 

In the next section we propose an algorithm that entirely eliminates
the combination step~\eqref{MFsolution}, and allows a 2-processor parallel 
scheme.

\section{New Two Filter Block Tridiagonal Algorithm}
\label{sec:NewAlgo}

In this section, we take advantage of the combined insight from forward and 
backward algorithms (FBT and BBT), to propose an efficient parallelized algorithm 
for block tridiagonal systems. In particular, in view of the analysis
regarding the RTS and M algorithms, it is natural to combine them by 
running them at the same time (using two processors). However, there is no need
to run them all the way through the data and then combine (as in the MF scheme)

Instead, the algorithms can work on the same matrix, meet in the middle, and exchange information,
obtaining the smoothed estimate for the point in the middle of the time series. 
At this point, each need only back-substitute into its own (independent) linear system
which is half the size of the original. 

We illustrate on a $6 \times 6$ system of type~\eqref{BlockTridiagonalEquation}, 
where steps 1,2 of FBT and BBT have been simultaneously
applied (for three time points each). The resulting equivalent linear system has form
\begin{equation}
\label{HalfFBpre}
\left(
\begin{matrix}
d^f_1 & c_2^T & 0 & 0 & 0 & 0\\
0 & d^f_2 & c_3^T & 0 & 0 & 0\\
0 &  0& d^f_3 & c_4^T & 0 & 0\\
0 &  0& c_4 & d^b_4 & 0 & 0\\
0 &  0& 0 & c_5 & d^b_5 & 0\\
0 &  0& 0 & 0 & c_6 & d^b_6\\
\end{matrix}
\right)
\left(
\begin{matrix}
e_1 \\ e_2 \\ e_3 \\ e_4\\ e_5\\ e_6
\end{matrix}
\right)
=
\left(
\begin{matrix}
s_1^f\\s_2^f\\s_3^f\\s_4^b\\s_5^b\\s_6^b
\end{matrix}
\right)
\end{equation}
The superscripts $f$ and $b$ denote whether the variables correspond to steps 1,2 of FBT or BBT, respectively. 
At this juncture, we have a choice of which algorithm is used to decouple the system. Supposing the BBT 
takes the lead,  
the requisite row operations are 
\begin{equation}
\label{middleOp}
\begin{aligned}
\text{row}_3 &\gets \text{row}_3 - c_4^T (d_4^b)^{-1}\text{row}_4 \\
\text{row}_4 &\gets \text{row}_4 - c_4 (\hat d_3^f)^{-1}\text{row}_3 
\end{aligned}
\end{equation}
where 
\[
\begin{aligned}
\hat d_3^f &= d_3^f - c_4^T (d_4^b)^{-1}c_4\\
\hat s_3^f & = s_3^f - c_4^T(d_4^b)^{-1}s_4\\
\hat s_4^b & = s_4^f - c_4(\hat d_3^f)^{-1}\hat s_3^f\\
\end{aligned}
\]
After these operations, the equivalent linear system is given by 
\begin{equation}
\label{HalfFBpost}
\left(
\begin{matrix}
d^f_1 & c_2^T & 0 & 0 & 0 & 0\\
0 & d^f_2 & c_3^T & 0 & 0 & 0\\
0 &  0& \hat d^f_3 & 0 & 0 & 0\\
0 &  0& 0 & d^b_4 & 0 & 0\\
0 &  0& 0 & c_5 & d^b_5 & 0\\
0 &  0& 0 & 0 & c_6 & d^b_6\\
\end{matrix}
\right)
\left(
\begin{matrix}
e_1 \\ e_2 \\ e_3 \\ e_4\\ e_5\\ e_6
\end{matrix}
\right)
=
\left(
\begin{matrix}
s_1^f\\s_2^f\\\hat s_3^f\\\hat s_4^b\\s_5^b\\s_6^b
\end{matrix}
\right)
\end{equation}
This system has two independent (block-diagonal) components (note that $\hat x_3$ and $\hat x_4$ are immediately available), 
and the algorithm can be finished in parallel, using steps 3,4 of algorithms FBT and BBT. 

Now that the idea is clear, we establish the formal algorithm. Note the similarities and differences to the MF smoother 
in Algorithm~\ref{algoDiag}. For $a \in \mathbb{R}_+$, we use the notation $[a]$ to denote the floor function. 

\begin{algorithm}[New two filter block tridiagonal algorithm]
\label{algoPBT}
The inputs to this algorithm are 
\( \{ c_k \} \),
\( \{ b_k \} \),
and
\( \{ r_k \} \).
The output is a sequence \( \{ e_k \} \)
that solves equation~(\ref{BlockTridiagonalEquation}).
\end{algorithm}
\begin{enumerate}

\item
\begin{itemize}
\item
\T{Set} \( d_1^f = b_1  \). \T{Set} \( s_1^f = r_1 \).

\T{For} \( k = 2 \) \T{To} \( [N/2] - 1 \) \T{:}

\begin{itemize}
\item
\T{Set} \( d_k^f = b_k - c_{k} (d_{k-1}^f)^{-1} c_{k}^\R{T} \).
\item
\T{Set} \( s_k^f = r_k - c_{k} (d_{k-1}^f)^{-1} s_{k-1} \).
\end{itemize}

\item
\T{Set} \( d_N^b = b_N  \) and \( s_N^b = r_N \).

\T{For} \( k = N, \ldots , [N/2] \),
\begin{itemize}
\item
\T{Set} 
\( d_k^b = b_k - c_{k+1}^\R{T} (d_{k+1}^b)^{-1} c_{k+1} \).
\item
\T{Set}
\( s_k^b = r_k - c_{k+1}^\R{T} (d_{k+1}^b)^{-1} s_{k+1} \).
\end{itemize}
\end{itemize}

\item 
\T{Set}
\[
\begin{aligned}
\hat d_{[N/2]-1}^f &= d_{[N/2]-1}^f - c_{[N/2]}^T (d_{[N/2]}^b)^{-1}c_{[N/2]}\\
\hat s_{[N/2]-1}^f & = s_{[N/2]-1}^f - c_{[N/2]}^T(d_{[N/2]}^b)^{-1}s_{[N/2]}\\
\hat s_{[N/2]}^b & = s_{[N/2]}^f - c_{[N/2]}(\hat d_{[N/2]-1}^f)^{-1}s_{[N/2]-1}^f
\end{aligned}
\]

\item
\begin{itemize}
\item
\T{Set} \( e_{[N/2]-1} = (\hat d_{[N/2]-1})^{-1} \hat s_{[N/2]-1} \).\\
\T{For} \( k = {[N/2]-2}\) \T{To} \( 1 \) \T{:}
\begin{itemize}
\item
\T{Set} \( e_k = (\hat d_k^f)^{-1} ( \hat s_k^f - c_{k+1}^\R{T} e_{k+1} ) \).
\end{itemize}

\item
\T{Set} \( e_{[N/2]} = (\hat d_{[N/2]}^b)^{-1} \hat s_{[N/2]}^b \).

\T{For} \( k = {[N/2+1]} , \ldots , N \),
\begin{itemize}
\item
\T{Set} \( e_k = (\hat d_k^b)^{-1} ( \hat s_k^b - c_k e_{k-1} ) \).
\end{itemize}
\end{itemize}

\end{enumerate}

Steps 1 and 3 of this algorithm can be done in parallel on two processors. The middle 
step 2 is the only time the processors need to exchange information. This algorithm therefore
has the desirable parallel features of Algorithm~\ref{algoDiag}, but in addition, 
step 3 of Algorithm~\ref{algoDiag} is not required. 

We now present a theorem that shows that when the overall system is well-conditioned, 
Algorithm~\ref{algoPBT} has the same stability results as FBT and BBT. 

\begin{theorem}
\label{thm:PBTstability}

Consider any block tridiagonal system $A\in \mathbb{R}^{Nn}$ 
of form~\eqref{BlockTridiagonalEquation}. 
and suppose we are given the bounds 
$\alpha_L$ and $\alpha_U$ for the lower and 
upper bounds of the eigenvalues of this system
\begin{equation}
\label{eigenvalueBounds}
0 < \alpha_L \leq \lambda_{\min}(A) \leq \lambda_{\max} (A) \leq \alpha_U\;.
\end{equation}
Then same bounds apply to all blocks $d_k^f$, $d_k^b$ of Algorithm~\ref{algoPBT}.
\end{theorem}

\begin{proof}
Going from~\eqref{HalfFBpre} to~\eqref{HalfFBpost} (steps 1,2 of Algorithm~\ref{algoPBT}),  
by the proof technique of Theorems~\ref{algorithmStability} and~\ref{algorithmStabilityBackward}
the eigenvalues of the lower block are bounded by the eigenvalues of the full system. 
In step 3, the matrix $\hat d^f_{[N/2]-1}$ created in operation~\eqref{middleOp} has the same property, 
by Theorem~\eqref{algorithmStabilityBackward}. The same results holds regardless of whether 
the FBT or BBT algorithms performs the `combination' step 3. 
\end{proof}


\section{Conclusions}


In this paper, we have characterized the numerical stability of block tridiagonal systems
that arise in Kalman smoothing, see theorems~\ref{simpleBounds} and~\ref{singularTheorem}.
The analysis revealed that last blocks of the system have a strong effect on the system overall,
and that the system may be numerically stabilized by changing these blocks. 
In the Kalman smoothing context, the stability conditions 
in theorems~\ref{simpleBounds} and~\ref{singularTheorem}
do not require $a_k$ in~\eqref{triSys} to be invertible. In fact, they may be singular,
which means the process matrices $G_k$, or derivatives of nonlinear process functions 
$g_k^{(1)}$, do not have to be invertible. 

We then showed that any well-conditioned symmetric block tridiagonal system 
can be solved in a stable manner with the FBT, which is equivalent to RTS  
in the Kalman smoothing context. 
We also showed that the forgotten M smoother, i.e. Algorithm A in~\cite{Mayne1966}, is equivalent 
to the BBT algorithm. BBT shares the stability properties of FBT for well-conditioned systems, 
but is unique among the algorithms considered because it 
can remain stable even when applied to ill-conditioned systems. 

We also characterized the standard MF scheme based on the two-filter formula, 
showing how it solves the underlying system, and proving that it has the same numerical
properties of RTS but does not have the numerical feature of M. 

Finally, we designed a new stable parallel algorithm, which is more efficient 
than MF and has the same stability guarantee as RTS for well-conditioned
systems. The proposed algorithm is parallel, simple to implement, and stable for well-conditioned 
systems. It is also more efficient than smoothers based on two independent filters, 
since these approaches require full parallel passes and then a combination step. 
Such a combination step is unnecessary here --- after the parallel forward-backward pass, 
we are done.

Taken together, these results provide insight into both numerical stability and algorithms for 
Kalman smoothing systems. In this regard, it is worth stressing that block tridiagonal systems 
arise not only in the classic linear Gaussian setting, but also in many new applications,  
including nonlinear process/measurement models, 
robust penalties, and inequality constraints, 
e.g. see \cite{Durovic1999,Cipra1997,AravkinIeee2011,AravkinBurkePillonetto2013,Farahmand2011}. 
As a result, the theory and the
new solvers of general block tridiagonal systems  
developed in this paper have immediate application to a wide range of problems.

\bibliographystyle{plain}
\bibliography{blk_tri}

\end{document}